\documentclass[12pt]{article}

\usepackage{latexsym,amsmath,amscd,amssymb,framed,graphics,color}
\usepackage{enumerate}
\usepackage{graphicx}
\usepackage{url}
\usepackage[all]{xy}
\usepackage[colorlinks]{hyperref}
\usepackage[square,authoryear]{natbib}

\newcommand{\bfi}{\bfseries\itshape}
\newcommand{\rem}[1]{}

\makeatletter

\@addtoreset{figure}{section}
\def\thefigure{\thesection.\@arabic\c@figure}
\def\fps@figure{h, t}
\@addtoreset{table}{bsection}
\def\thetable{\thesection.\@arabic\c@table}
\def\fps@table{h, t}
\@addtoreset{equation}{section}

\makeatother

\newcommand \al{\alpha}
\newcommand\be{\beta}

\newcommand\de{\delta}
\newcommand\ep{\varepsilon}

\newcommand\et{\eta}

\newcommand\la{\lambda}
\newcommand\rh{\rho}

\newcommand\si{\sigma}
\newcommand\ta{\tau}
\newcommand\ph{\varphi}

\newcommand\ps{\psi}

\newcommand\Ga{\Gamma}

\newcommand\Om{\Omega}

\newcommand\resp{resp.\ }
\newcommand\ie{i.e.\ }

\newcommand\oo{{\infty}}

\renewcommand\o{\circ}

\newcommand\x{\times}
\newcommand\on{\operatorname}

\newcommand\Emb{\on{Emb}}

\newcommand\Vol{\on{Vol}}

\newcommand\Diff{\on{Diff}}
\newcommand\Gr{\on{Gr}}

\newcommand\vol{\on{vol}}

\newcommand\pa{\partial}

\newcommand\dd{{\mathbf d}}

\newcommand\X{\mathfrak X}

\newcommand\V{\mathcal V}
\newcommand\U{\mathcal U}

\newcommand\RR{\mathbb R}
\newenvironment{proof}[1][Proof]{\noindent\textbf{#1.} }{\ \rule{0.5em}{0.5em}}

\def\XXint#1#2#3{{\setbox0=\hbox{$#1{#2#3}{\int}$ }
\vcenter{\hbox{$#2#3$ }}\kern-.5\wd0}}

\begin{document}

\newtheorem{theorem}{Theorem}[section]
\newtheorem{definition}[theorem]{Definition}
\newtheorem{lemma}[theorem]{Lemma}
\newtheorem{remark}[theorem]{Remark}
\newtheorem{proposition}[theorem]{Proposition}
\newtheorem{corollary}[theorem]{Corollary}
\newtheorem{example}[theorem]{Example}

\def\below#1#2{\mathrel{\mathop{#1}\limits_{#2}}}



\title{Principal bundles of embeddings and nonlinear Grassmannians}
\author{Fran\c{c}ois Gay-Balmaz and Cornelia Vizman}

\date{ }
\maketitle
\makeatother
\maketitle

\begin{abstract}
{We present several principal bundles of embeddings of compact manifolds
(with or without boundary) whose base manifolds are nonlinear Grassmannians. We study their infinite dimensional differential manifold structure in the Fr\'echet category. This study is motivated by the occurrence of such objects in the geometric Lagrangian formulation of free boundary continuum mechanics and in the study of the associated infinite dimensional dual pairs structures.}
\end{abstract}

\section{Introduction}

In this paper, we consider several infinite dimensional principal bundles whose total space is a set of embeddings and whose structure group is given by a group of reparametrizing diffeomorphisms. The base space of these principal bundles are certain classes of nonlinear Grassmannians, that is, sets of submanifolds. Such infinite dimensional objects appear naturally, for example, in the geometric Lagrangian and Hamiltonian formulations of free boundary continuum mechanics and the associated processes of reduction by symmetries, see e.g., \cite{LeMaMoRa1986}, \cite{GBMaRa2012}.

The present paper grew out of the necessity of a rigorous differential geometric description
of some infinite dimensional objects that naturally emerged in the study of dual pairs structures related to fluid dynamics,
\cite{GBVi2013} and \cite{GBVi2014}. Let us mention that a dual pair is an important concept in Poisson geometry. For example, it allows to obtain, in some cases, additional informations about symplectic reduced spaces, see e.g.,  \cite{BW}.

The paper \cite{GBVi2013} presents a dual pair structure for free boundary fluids. This dual pair is defined on the cotangent bundle $T^* \operatorname{Emb}_{\vol}(S,M)$ of the manifold $ \operatorname{Emb}_{\vol}(S,M)$ of all embeddings of a compact volume manifold $(S,\mu)$ with boundary $\pa S$ into a volume manifold $(M,\mu_M)$ without boundary, both having the same dimension. Its study is based on the existence of a Fr\'echet differential manifold structure on the infinite dimensional manifolds involved, such as $ \operatorname{Emb}_{\vol}(S,M)$, $ \operatorname{Emb}_{\vol}(S,M)/ \operatorname{Diff}_{\vol}(S)$, $(T  \operatorname{Emb}_{\vol}(S,M))/ \operatorname{Diff}_{\vol}(S)$. 

The paper \cite{GBVi2014} makes use of the dual pair of momentum maps for ideal fluids found in \cite{MaWe83} to identify and describe prequantizable coadjoint orbits of the group of Hamiltonian diffeomorphisms. This dual pair is defined on the manifold of embeddings $\Emb(S,M)$, where $M$ is a prequantizable symplectic manifold. The result is obtained by applying symplectic reduction with respect to the momentum map associated to the action of the group of volume preserving diffeomorphisms 
$ \operatorname{Diff}_{ \vol }(S)$, and yields the submanifold of the nonlinear Grassmannian 
$\Gr^{S,\mu}(M)$ of volume submanifolds of $M$  of type $(S,\mu)$ that consists of isotropic submanifolds of $M$. Here again, the study is based on the existence of a Fr\'echet differential manifold structure on the principal bundles and nonlinear Grassmannians involved.

\medskip

Our study extends or complements earlier work on the infinite dimensional manifold structures on principal bundles of embeddings, such as e.g., \cite{Ha1982}, \cite{BF}, \cite{KrMi97}, \cite{Molitor08}, \cite{Molitor12}.


\medskip


Below we recall some basic definitions related to calculus on Fr\'echet manifolds, see e.g. \cite{Ha1982}. A more general calculus (the convenient calculus) is developed in \cite{KrMi97}, to deal with manifold modeled on a more general class of topological vector spaces than Fr\'echet spaces. In this paper we will only need the Fr\'echet setting, but since the two approaches coincide in this case, we will use both references.

\paragraph{Review on Fr\'echet manifolds.}
Recall that a \textit{Fr\'echet space} is a topological vector space whose topology arises from a countable collection of seminorms, and is Hausdorff and complete. For example, the space $ \Gamma (V)$ of all smooth (i.e., $C^\infty$) sections of a vector bundle $V \rightarrow S$ over a smooth compact finite dimensional manifold $S$ (possibly with smooth boundary) is a Fr\'echet space. A countable collection of seminorms is obtained by choosing Riemannian metrics and connections on the vector bundles $TS$ and $V$ and defining
\[
\| \sigma \|_n:=\sum_{j=0}^n\sup_{ s \in S} | \nabla ^j \sigma (s)|,\;\;n=0,1,2,...,
\]
where $ \nabla ^j $ denotes the $j^{th}$ covariant derivative of a section $ \sigma \in \Gamma (V)$. In particular, the space $C^\infty(S,V)$ of smooth maps on $S$ with values in a finite dimensional vector space $V$ is a Fr\'echet space.

Let $E,F$ be Fr\'echet spaces and $U$ be an open subset of $E$. A continuous map $f: U \subset E \rightarrow F$ is \textit{differentiable at $x \in U$ in the direction $e$} if the limit
\[
Df(x) \cdot e:=\lim_{h \rightarrow 0} \frac{f(x+he)-f(x)}{h}
\]
exists. We say that the map $f$ is \textit{continuously differentiable} (or $C^1$) if the limit exists for all $ x \in U$ and all $e \in E$ and if the map $ Df:( U \subset E) \times E \rightarrow F$, $(x,e) \mapsto Df(x) \cdot e$ is continuous (jointly as a function on a subset of the product).
Note that if the Fr\'echet spaces $E$ and $F$ turn out to be Banach spaces, this definition of $C^1$ map does not recover the usual definition for Banach spaces since the continuity requirement is weaker.
The concept of smooth ($C^\infty$) maps between Fr\'echet space is defined from the $C^1$ case in the usual iterative way.

A {\it Fr\'echet manifold} modeled on the Fr\'echet space $E$ is a Hausdorff space $M$ covered by charts $(U,\varphi )$ where $U\subseteq M$ is open and $\varphi :U\to \varphi (U)\subseteq E$ is a homeomorphism, such that for any two charts $(U,\varphi )$ and $(U',\varphi ')$, the coordinate changes
$$
\varphi '\o \varphi ^{-1}|_{\varphi (U\cap U')}:\varphi (U\cap U')\to\varphi '(U\cap U')
$$
are smooth maps between Fre\'chet spaces. An example of Fr\'echet manifold is provided by the set $C^\infty(S,M)$ of all smooth maps on a compact finite dimensional manifold $S$ (possibly with smooth boundary) into a finite dimensional manifold $M$ without boundary.
The tangent space at $f$ is the Fr\'echet space $T_f C^\infty(S,M)= \Gamma (f ^\ast TM)$ of all smooth sections of the pull-back vector bundle $f ^\ast TM \rightarrow S$.

A subset $N\subseteq M$ is called a {\it smooth Fr\'echet submanifold} of the Fr\'echet manifold $M$ if there exists a closed subspace $F$ of the Fr\'echet space $E$ and $N$ is covered by charts $(U,\varphi )$ of $M$ such that $\varphi (U\cap N)=\varphi (U)\cap F$. These kinds of charts are called {\it submanifold charts}.
The submanifold $N$ is called a {\it splitting submanifold} if 
the closed subspace $F\subset E$ is complemented, \ie there exists a subspace $G\subset E$
such that the addition map $F\x G\to E$ is a topological isomorphism.

We shall make use of the following regular value theorem \cite{NW} that follows from an implicit function theorem
\cite{Gloeckner}. Let $f:M\to Q$ be a smooth map from a Fr\'echet manifold $M$ into a Banach manifold $Q$. Fix $ q _0 \in Q$ and suppose that for all $m \in M$, with $f(m)= q _0 $, the tangent map $T_mf:T_mM \rightarrow  T_{ q _0 }Q$ is surjective and the subspace $\operatorname{ker}T_mf$ is complemented, i.e., $f$ is a \textit{submersion}. Then $f^{-1}(q_0)$ is a splitting submanifold of $M$ and $T_m ( f ^{-1} ( q _0 ))= \operatorname{ker}(T_mf)$.
The hypotheses on the tangent map are equivalent to the existence, for all $m$ with $f(m)= q _0 $, of a right inverse $ \sigma _m :T_ {q _0 } Q\rightarrow T_mM$ of $T_mf$.  Note that if $Q$ is finite dimensional, then $\operatorname{ker}T_mf$ is necessarily complemented.
\color{black}

Let $G$ be a \textit{Fr\'echet Lie group}, i.e., a Fr\'echet manifold with a group structure such that the group multiplication and the inversion are smooth maps. A principal $G$-bundle in the Fr\'echet category consists of a Fr\'echet manifold $Q$, a (right) action $ G \times Q \rightarrow Q$, and a submersion $ \pi : Q \rightarrow M$, where $M$ is another Fr\'echet manifold, such that for each $m \in M$, we can find a neighborhood $U$ of $m$ and a diffeomorphism $ \Psi  :\pi ^{-1} (U) \rightarrow G \times U$ such that $(i)$ the action of $G$ on $Q$ corresponds to the action on $G \times U$ on the first factor by right translation, and $(ii)$ the projection $ \pi $ corresponds to the projection of $G \times U$ onto the second factor.
 
\paragraph{Manifolds and submanifolds with boundary.} Recall that a \textit{differentiable manifold $S$ with boundary} is defined by requiring that its charts $ \varphi $ are bijections from a subset $U$ of $S$ to an open subset of the
closed half-plane $\mathbb{R}  ^n _+:= \{(x_1, ..., x_n) \in \mathbb{R}  ^n \mid x _n \geq 0\}$, endowed with the topology induced from $ \mathbb{R}  ^n $. Thus, the overlap maps are required to be
diffeomorphisms between open subsets of $\mathbb{R}  ^n _+$. Recall that a map $f : U \rightarrow  V$ between
open subsets $U, V$ of $\mathbb{R}  ^n _+$ is of class $C ^k $ if for each point $x \in  U$ there exist open
neighborhoods $U_1$ of $x$ and $V_1$ of $f(x)$ in $\mathbb{R}  ^n $ and a map $f_1 : U_1 \rightarrow  V_1$ of class $C^k$ such that $f|U \cap U_1 = f_1|U \cap U_1$. We then define $ \mathbf{D} ^i f(x) := \mathbf{D}^i  f_1(x)$, for $x \in  U \cap U_1$ and for all $i = 1, ..., k$. This definition is independent of the choice of $f_1$.

We now recall the definition of a submanifold with boundary (see e.g. \cite[\S4]{H76}). A subset $P\subset \mathbb{R}  ^n $ is a \textit{submanifold of $ \mathbb{R}  ^n$ with boundary} if each point of $P$ belongs to the domain of a chart $\varphi : W \rightarrow \mathbb{R}  ^n $ of $ \mathbb{R}  ^n $ such that
\[
\varphi (W \cap P)= \varphi (W)\cap (\mathbb{R}  ^k_+\times \{0\} )\subset \mathbb{R}  ^n.
\]
Let $M$ be a manifold (with or without boundary). A subset $N \subset M$ is a \textit{submanifold of $M$ with boundary} if each point of $N$ belongs to the domain of a chart $ \varphi :U \rightarrow \mathbb{R}  ^n $ of $M$ such that $ \varphi (U \cap N)$ is a submanifold of $\mathbb{R}^n $ with boundary.

\paragraph{Immersions and embeddings.} 
Let $S$ be a manifold with boundary and $M$ a manifold without boundary. A map $f:S \rightarrow M$ is of \textit{class $C ^k$} if each of its local representations $f_{loc}: U \rightarrow V$ is of class $C ^k $ in the sense recalled above, where $U$ is an open subset of $ \mathbb{R}  ^n _+$ and $V$ is an open subset of $\mathbb{R}  ^m $.
A smooth map $f:S \rightarrow M$ is an \textit{immersion} if its tangent map $T_sf:T_sS \rightarrow T_{f (s) }M$ is injective for all $s$, where we recall that the tangent space to $S$ at $s \in \partial S$ has the same dimension as the tangent spaces at interior points.
A smooth map $f:S \rightarrow M$ is an \textit{embedding} if it is an injective immersion and a homeomorphism onto its image (with the subspace topology). In this case the image $N:=f(S)$ is a smooth submanifold of $M$ with boundary and $ \partial N= f( \partial S)$.
If $S$ is compact then an injective immersion $f:S \rightarrow M$ is necessarily an embedding.

Given a compact manifold $S$ with or without boundary and $M$ a manifold without boundary, such that $ \operatorname{dim}S\leq \operatorname{dim}M$, then the set  $\Emb(S,M)$ of all embeddings is an open submanifold of the manifold of all smooth maps $C^\oo(S,M)$ (see e.g. \cite{Michor} Proposition 5.3).
{Similarly, the group of diffeomorphisms $\Diff(S)$
is an open submanifold of $C^\oo(S,S)$. The group operations are smooth,
so $\Diff(S)$ carries naturally a Lie group structure.}

\paragraph{Nonlinear Grassmannians.} Let $M$ be a smooth finite dimensional manifold without boundary and let $S$ be a finite dimensional compact manifold possibly with boundary. The \textit{nonlinear Grassmannian of submanifolds of $M$ of type $S$} is the set $ \operatorname{Gr}^S(M)$ defined by
\[
 \operatorname{Gr}^S(M):= \{ N \subset M\mid \text{$N$ is a submanifold of $M$ diffeomorphic to $S$}\}.
\] 
Of course, when $S$ has boundary, then $N$ is a submanifold of $M$ with boundary, in the sense recalled above.

Given $N \in \operatorname{Gr}^S(M)$, by definition there exists $f \in \operatorname{Emb}(S,M)$ such that $N=f(S)$. The embedding $f$ is unique up to the composition on the right by a diffeomorphism in $ \operatorname{Diff}(S)$. So we get the bijection
\[
\operatorname{Emb}(S,M)/ \operatorname{Diff}(S) \longleftrightarrow   \operatorname{Gr}^S(M), \quad [f] \mapsto f(S).
\]

Two other nonlinear Grassmannians will emerge naturally in this paper as bases of principal bundles. One is the nonlinear Grassmannian $\operatorname{Gr}^S_0 (M)$ of submanifolds of $M$ diffeomorphic to $S$, 
with the same total volume as $S$. It arises as the quotient space $ \operatorname{Emb}_0(S,M)/ \operatorname{Diff}(S)$, where $ \operatorname{Emb}_0(S,M)$ is the set of all embeddings that preserve the total volume. It also arises as the quotient space $ \operatorname{Emb}_{\vol}(S,M)/ \operatorname{Diff}_{\vol}(S)$,  of volume preserving embeddings by volume preserving diffeomorphisms. Another one is the nonlinear Grassmannian $\operatorname{Gr}^{S, \mu }(M)$ of volume submanifolds of type $(S, \mu )$, where $ \mu $ is a volume form on $S$. It arises as the quotient space $ \operatorname{Emb}(S,M)/ \operatorname{Diff}_{\vol}(S)$.  

\paragraph{Plan of the paper.}

In Section \ref{2.1}, we show that $\operatorname{Emb}(S,M) \rightarrow \operatorname{Gr}^S(M)$ is a smooth Fr\'echet principal $ \operatorname{Diff}(S)$-bundle, for any compact manifold $S$, with $ \operatorname{dim}S\leq \operatorname{dim}M$, with or without boundary. In particular, this shows that $\operatorname{Emb}(S,M)$ and $\operatorname{Gr}^S(M)$ are Fr\'echet manifolds. This extends a result already known in the case $ \partial S=\varnothing$, see \cite{KrMi97}. The main difficulty of the proof is the construction of manifold charts for $\operatorname{Gr}^S(M)$, when $S$ has boundary. In Section \ref{s.totvol}, we show that $\operatorname{Emb}_0 (S,M) \rightarrow \operatorname{Gr}^S_0 (M)$ is a smooth Fr\'echet principal $ \operatorname{Diff}(S)$-bundle. This is done by using the regular value theorem recalled above, valid in the Fr\'echet context. In Section \ref{s2} we show that $ \operatorname{Gr}^{S, \mu }(M)$ is a Fr\'echet manifold by identifying it with an associated bundle to $ \operatorname{Emb}(S,M)$. Using Moser's decomposition of the diffeomorphism group $ \operatorname{Diff}(S)$, we then prove that $\operatorname{Emb}(S,M) \rightarrow \operatorname{Gr}^{S, \mu }(M)$ is a smooth Fr\'echet principal $ \operatorname{Diff}_{\vol}(S)$-bundle.
Finally, in Section \ref{s5}, we show that $\operatorname{Emb}_{\vol}(S,M) \rightarrow \operatorname{Gr}^S_0(M)$ is a smooth Fr\'echet principal $ \operatorname{Diff}_{\vol}(S)$-bundle, by identifying it with a pull-back bundle of $\operatorname{Emb}(S,M) \rightarrow \operatorname{Gr}^{S, \mu }(M)$. In particular, this shows that $\operatorname{Emb}_{\vol}(S,M)$ is a Fr\'echet manifold. This extends a result known for the set of volume preserving embeddings with nowhere vanishing mean curvature, in the particular case $\partial S=\varnothing$, see \cite{Molitor12}.

\smallskip

In this paper, all the finite dimensional manifolds are assumed to be connected and orientable.

\paragraph{Acknowledgements.} This work was partially supported by a grant of the Romanian National Authority for  Scientific  Research, CNCS UEFISCDI, project number PN-II-ID-PCE-2011-3-0921.


\section{Principal $\Diff(S)$-bundle structure on embeddings}\label{2.1}

When the manifold $S$ has no boundary, it is known that $ \operatorname{Emb}(S,M)$ and $ \operatorname{Gr}^S(M)$ are Fr\'echet manifolds and that
\[
\pi:\Emb(S,M)\rightarrow  \Gr^S(M), \quad f \mapsto \pi (f):=f(S)
\]
is a $ \operatorname{Diff}(S)$-principal Fr\'echet bundle. In Section \S\ref{s2.2}, we shall extend this result to  the case when the manifold $S$ has a non empty boundary. Since the proof in that case is considerably more involved, we first sketch below in Section \S\ref{s2.1} the proof in the boundaryless case, following \cite{KrMi97}. 
A slightly different approach to the principal bundle of embeddings
is taken in \cite{Ha1982} and \cite{Molitor08}.

Note that if we take as structure group the connected component $\Diff_+(S)$ of orientation preserving diffeomorphisms, then the base space is the nonlinear Grassmannian of oriented submanifolds of $M$ of type $S$.

\subsection{The boundaryless case}\label{s2.1} 

Assume that $S$ is compact without boundary. We restrict to the case $ \operatorname{dim}S< \operatorname{dim}M$,
otherwise every embedding would be a diffeomorphism.

Let us fix $ f _0 \in \operatorname{Emb}(S,M)$ and $N_0:=f_0(S)$. We endow $M$ with a Riemannian metric $g$, so that the normal vector bundle $TM|_{N_0}/TN_0\rightarrow N_0$ can be identified with the orthogonal subbundle $p:TN_0^{\perp}\subset TM|_{N_0}\rightarrow N _0$.  
A normal tubular neighborhood $U\subset M$ of $N_0$ in $M$ can be built with the exponential map associated to $g$, namely with the diffeomorphism 
\begin{equation}\label{has}
\ta:=\exp_{N_0}:V\subset TN_0^\perp\to U\subset M,
\end{equation}
where $V $ is an open neighborhood of the zero section of $p:TN_0^{\perp}\rightarrow N _0$.
This diffeomorphism coincides with the identity on $N_0$.

We denote the set of $V$-valued sections of the vector bundle $TN_0^\perp$ by
\begin{equation}\label{gamave}
 \Gamma _V( TN_0^\perp ):=\{ \sigma \in \Gamma (TN_0^\perp )\mid \sigma (N_0) \subset V\}.
\end{equation} 
It is an open subset of the Fr\'echet vector space $\Gamma ( TN_0^\perp)$.

\begin{definition} An embedding of $S$ into $M$ built with a normal vector field $\si\in\Ga_{V} (TN_0^\perp)$ via the tubular neighborhood diffeomorphism $\ta$ as
\[
f^\perp:=\ta\o\si\o f_0:S \rightarrow M
\]
will be called a {\bfi normal embedding relative to $N_0$}.
\end{definition}

Note that since $ \sigma $ is a section, it is an injective immersion and hence an embedding since $S$ is compact. Since $ \tau $ and $f_0:S\to N_0$ are
diffeomorphisms, the composition $ \tau \circ \sigma \o f_0:S \rightarrow M$ is an embedding.

\medskip

We denote by $ \mathcal{U} \subset \operatorname{Gr}^S(M)$ the set of all submanifolds $N\in \operatorname{Gr}^S(M)$ obtained via normal embeddings relative to $N_0$, constructed from sections $ \sigma \in \Gamma _V( TN_0^\perp)$.
Given a submanifold $N  \in \mathcal{U} $, there is a unique normal embedding $f^\perp$ with image $N$.
It can be recovered from an arbitrary embedding $f:S\to M$ with $f(S)=N$, by writing $f^\perp=f\o\ps_f^{-1}$, where 
\begin{equation}\label{psif}
\ps_f= f_0^{-1}\o p\o \ta^{-1}\o f\in\Diff(S).
\end{equation}
The embedding $f^\perp$ does not depend on the choice of the embedding $f$, 
because starting with another embedding of $S$ onto $N\subset M$, $f\o\ps$
with $\ps\in\Diff(S)$,
we get the same normal embedding $f^\perp$ since $\ps_{f\o\ps}=\ps_f\o\ps$.

\paragraph{Fr\'echet manifold structure on $\Gr^S(M)$.} A chart around $N_0$ is defined on the subset $\U\subset \operatorname{Gr}^S(M)$ by
\begin{equation}\label{chi0}
\chi : \mathcal{U} \rightarrow \Gamma _V( TN_0^\perp ), \quad \chi(N):=\ta^{-1}\o f^\perp\o f_0^{-1},
\end{equation}
where $f^\perp$ the unique normal embedding relative to $N_0$ such that $f^\perp(S)=N$. The inverse reads $\chi^{-1}(\si)=(\ta\o\si)(N_0)$.

Consider two submanifolds $N _i , N _j \in \operatorname{Gr}^S(M)$, let $\U_i$, \resp $\U_j$, denote the set of all submanifolds of $M$ that are images of normal embeddings relative to $N_i$, \resp $N_j$,
and let $\chi_i$, \resp $\chi_j$, be the corresponding charts.
Then the chart change $\chi_j\o\chi_i^{-1}$ reads
\[
{ \chi _i ( \mathcal{U} _i \cap \mathcal{U} _j )} \subset \Ga(TN_i^\perp) \rightarrow { \chi _j( \mathcal{U} _i \cap \mathcal{U} _j )}  \subset \Ga(TN_j^\perp),
\quad \sigma \mapsto \ta_j^{-1}\o(\ta_i\o\si\o f_0^{-1})^{\perp_j}\o f_0^{-1},
\]
where $f^{\perp_j}$ denotes the normal embedding relative to $N_j$
associated to an embedding $f$.
Notice that the embedding  $\ta_i\o\si\o f_0^{-1}\in\Emb(S,M)$, 
with image $\chi_i^{-1}(\si)$, is a normal embedding 
relative to $N_i$, but not necessarily relative to $N_j$.

The fact that the chart changes are smooth maps between Fr\'echet spaces follows from the smoothness of the composition and inversion maps, see e.g. Corollary 3.13 and Theorem 43.1 in \cite{KrMi97}. Indeed, the map $f\in \operatorname{Emb}(S,M) \mapsto \ps_f=f_0^{-1}\o p\o\ta^{-1}\o f\in \operatorname{Diff}(S) $ is smooth, so the map $f\in \operatorname{Emb}(S,M)\mapsto f^\perp=f\o\ps_f^{-1}\in \operatorname{Emb}(S,M)$ is also smooth, which shows that $\si\mapsto \ta_j^{-1}\o(\ta_i\o\si\o f_0^{-1})^{\perp_j}\o f_0^{-1}$ is smooth.

This proves that $ \operatorname{Gr}^S(M)$ is a smooth Fr\'echet manifold. The connected component of an element $N \in   \operatorname{Gr}^S(M)$ is modeled on the Fr\'echet space $\Gamma (TN^\perp)$ of all smooth sections of the vector bundle $ TN^ \perp\rightarrow N$.
\color{black} 

Note that, while a Riemannian metric is involved in the construction of charts, it is not needed to write the tangent space to $ \operatorname{Gr}^S(M)$ at $N$, since we can make the identification $T_N \operatorname{Gr}^S(M)=\Gamma (TM|_N/TN)$.

\paragraph{Principal $\operatorname{Diff}(S)$-bundle structure on $\Emb(S,M)$.} {We now make use of the Fr\'echet manifold structures on  $ \operatorname{Diff}(S)$ and $\Gr^S(M)$ to show that $\pi:\Emb(S,M)\to\Gr^S(M)$, $\pi(f)=f(S)$ is a principal $\Diff(S)$-bundle.}

The local trivializations are built over the open sets $\U$
of all submanifolds of $M$ that are images of normal embeddings relative to $N_0$. They are given by
\[
\Psi:  \pi^{-1}(\mathcal{U}) \rightarrow \mathcal{U}\x\Diff(S), \quad 
\Psi(f):=( f(S) ,  \psi_f),
\]
where $ \psi_f= p \circ \ta^{-1} \circ f$ as above. The inverse is given by $\Psi ^{-1} ( N , \psi)= f^\perp \circ\psi$,
where $f^\perp$ is the unique normal embedding relative to $N_0$ such that $f^\perp(S)=N$.
Using the equality $( f^{\perp_i}) ^{-1} = (p _i \circ \tau _i ^{-1}) |_{N}$, we obtain that the transition functions are given by
\[
\psi_{ij}:\U_i\cap\U_j\to\Diff(S),\quad\psi_{ij}(N)=(f^{\perp_i})^{-1}\o f^{\perp_j}.
\]
and are smooth maps.


Note that the tangent map to the projection $ \pi : \operatorname{Emb}(S,M) \rightarrow \operatorname{Gr}^S(M)$ sends a vector field $v_f\in\Ga(f^*TM)$ to the orthogonal projection $(v_f\circ f ^{-1} )^\perp\in\Ga(TN^\perp)$, where $N=f(S)$ (see \cite{Molitor08} Corollary 2.3).
The vertical space $V_f \operatorname{Emb}(S,M)$ at $f$ thus consists of vector fields $ v _f \in \Gamma ( f ^\ast TN)$ such that $v _f (s) \in T_{f(s)}N$.
As before, a metric is not needed to write the tangent map, since we can write $T_f \pi (  v_f )=[v _f \circ f ^{-1} ]\in \Gamma (TM/TN)$.

\subsection{The case $ \partial S\ne \varnothing$}
\label{s2.2} 

In this section we assume that $S$ is compact with boundary and that $\dim S\le\dim M$. Note that, contrary to the boundaryless case, the situation $ \operatorname{dim}S=\operatorname{dim}M$ is interesting.

The construction of manifold charts for the nonlinear Grassmannian $ \operatorname{Gr}^S(M)$ when $S$ has a non empty boundary is considerably involved. 
We thus start with two instructive examples, before showing the result in general.

\subsubsection{The principal bundle $ \operatorname{Emb}([0,1], \mathbb{R}  ^2 )\rightarrow \operatorname{Gr}^{[0,1]}( \mathbb{R}  ^2 )$}

As an illustrative example and to introduce our notations, we consider the case $S=[0,1]$, $M= \mathbb{R}  ^2 $, with $ \partial S=\{0,1\}$. The construction will be made precise in \S\ref{gen_case} below.

We consider the embedding $f_0:[0,1] \rightarrow \mathbb{R}  ^2 $ given by $ f_0(s)=(s,0)$. Intuitively, elements of $ \operatorname{Gr}^{[0,1]}( \mathbb{R}  ^2 )$ in a neighborhood of $N_0=[0,1]\times \{0\}$ can be obtained by first deforming $N _0$ in a $y$-direction and then in a $x$-direction in $ \mathbb{R}  ^2 $. Therefore a normal embedding associated to $[0,1]$ should be uniquely determined from two numbers $ \sigma ^\dagger=( \sigma _0 ^\dagger, \sigma _1^\dagger) \in (- \varepsilon/2 , \varepsilon/2 ) ^2 $ giving the $x$-deformation, and from a function $ \sigma \in C^\infty([0,1], (-\delta , \delta ))$ giving the $y$-deformation.

To $ \sigma ^\dagger$ we associate an embedding $ \phi _{ \sigma ^\dagger}: N_0 \rightarrow \mathbb{R}  ^2 $ given by $ \phi _{ \sigma ^\dagger}(x,0)=((1+ \sigma _1 ^\dagger- \sigma _0 ^\dagger)x+ \sigma _0 ^\dagger, 0)$, sending $N_0= [0,1] \times \{0\}$ to $N_1=[\sigma _0 ^\dagger, 1+\sigma _1 ^\dagger]\times \{0\}\subset (- \varepsilon/2 , 1+ \varepsilon/2 ) \times \{0\}$. From $ \phi _{ \sigma ^\dagger}$ we can build the diffeomorphism $H_{ \sigma ^\dagger}:[0,1] \times \mathbb{R}  \rightarrow [\sigma _0 ^\dagger, 1+\sigma _1 ^\dagger]\times \mathbb{R}  $ defined by $H_{ \sigma ^\dagger}(x,y):= ((1+ \sigma _1 ^\dagger- \sigma _0 ^\dagger)x+ \sigma _0 ^\dagger, y)$ associated to $x$-deformation.
To the function $ \sigma $ is naturally associated the embedding, also denoted $ \sigma $, given by $ \sigma (x,0)=( x, \sigma (x))$
associated to $y$-deformation.
The normal embedding is therefore $f^\perp(s):= H_{ \sigma ^\dagger}( \sigma ( f _0 (s)))=((1+ \sigma _1 ^\dagger- \sigma _0 ^\dagger)s+ \sigma _0 ^\dagger, \sigma (s))$. Such embeddings can be used to construct charts for $\operatorname{Gr}^{[0,1]}( \mathbb{R}  ^2 )$, thereby proving that $ \operatorname{Gr}^{[0,1]}( \mathbb{R}  ^2 )$ is a Fr\'echet manifold modeled on the Fr\'echet vector space $ \mathbb{R}  ^2 \times C^\infty([0,1], \mathbb{R}  )$.

\subsubsection{The principal bundle $ \operatorname{Emb}\left( \overline{ \mathbb{D}  }, \mathbb{R}  ^3 \right) \rightarrow \operatorname{Gr}^{\overline{ \mathbb{D}  }}( \mathbb{R}  ^3 )$} 

We now consider the case of the closed unit disk $\overline{ \mathbb{D}}=\{(x, y) \in \mathbb{R}  ^2 : x ^2 + y ^2 \leq 1\}$, with $M= \mathbb{R}  ^3 $ and $ \partial \overline{ \mathbb{D}  }=S ^1 $.

We consider the embedding $f_0:\overline{ \mathbb{D}  }\rightarrow \mathbb{R}  ^3 $ given by $ f_0(x,y)=(x,y,0)$. Intuitively, elements of $ \operatorname{Gr}^{\overline{ \mathbb{D}  }}( \mathbb{R}  ^3 )$ in a neighborhood of $N_0=\overline{ \mathbb{D}  }\times \{0\}$ can be obtained by first deforming $N _0$ in a $z$-direction and then in a $(x,y)$-direction perpendicular to the boundary. Therefore a normal embedding associated to $\overline{ \mathbb{D}  }$ should be uniquely determined from two functions $ \sigma ^\dagger\in C^\infty(S ^1 , (- \varepsilon/2 , \varepsilon/2 ) ^2 )$ and  $ \sigma \in C^\infty(\overline{ \mathbb{D}  }, (-\delta , \delta ))$.

Let us use polar coordinates $(r, \theta )$ in the $(x,y)$ plane and fix a smooth function $\rho:\RR\to\RR$ which vanishes in a neighborhood of $(-\infty,-1]$, with $\rho(0)=1$ and bounded derivative: $0\le\rho'(t)\le 2$ for all $t\in\RR$. To the function $ \sigma ^\dagger$ we associate an embedding $ \phi _{ \sigma ^\dagger}: N_0 \rightarrow \mathbb{R}  ^3 $ given by $ \phi _{ \sigma ^\dagger}(r, \theta ,0)=(r+ \sigma ^\dagger( \theta ) \rho ((r-1)/ \varepsilon ), \theta , 0)$, which deforms the disk $N_0$ in a $(x,y)$-direction perpendicular to the boundary. This results in the manifold $N_1 \subset \overline{\mathbb{D}}_ \varepsilon  \times \{0\}$ with boundary $ \partial N _1 =\{(1+ \sigma ^\dagger (\theta ), \theta,0 ): \theta \in [0,2 \pi ]\}\subset S ^1 \times (- \varepsilon , \varepsilon ) \times \{0\}$. From $ \phi _{ \sigma ^\dagger}$, we build the embedding $H_{ \sigma ^\dagger}: \overline{\mathbb{D}} \times \mathbb{R}  \rightarrow N _1  \times \mathbb{R}  $ defined by $H_{ \sigma ^\dagger}(r, \theta , z):= ( \phi _{ \sigma ^\dagger}(r, \theta ),z)$.

To the function $ \sigma $ is naturally associated the embedding, also denoted $ \sigma $, given by $ \sigma (r, \theta ,0)=(r, \theta ,\sigma (r, \theta ))$
associated to $z$-deformations.
The normal embedding is therefore $f^\perp(r, \theta ):= H_{ \sigma ^\dagger}( \sigma ( f _0 (r, \theta )))= ( \phi _{ \sigma ^\dagger}(r, \theta ),\sigma ( r, \theta ))$. Such embeddings can be used to construct charts for $\operatorname{Gr}^{\overline{ \mathbb{D}  }}( \mathbb{R}  ^3 )$, thereby proving that $\operatorname{Gr}^{\overline{ \mathbb{D}  }}( \mathbb{R}  ^3 )$ is a Fr\'echet manifold modeled on the Fr\'echet vector space $C^\infty(S ^1 , \mathbb{R}  )\times C^\infty(\overline{ \mathbb{D}  }, \mathbb{R}  )$.

\subsubsection{The general case}\label{gen_case}

Given $N_0=f_0(S)$ a submanifold of $M$ of type $S$
(compact with boundary), we choose a Riemannian metric $g$ on $M$  
such that $N_0$ is a totally geodesic submanifold of $M$.
We consider the normal bundle $TN_0^\perp$ over $N_0$
and the line bundle 
$$
T(\pa N_0)^\dagger:=TN_0|_{\pa N_0}\cap T(\pa N_0)^\perp
$$
over $\pa N_0$. 
We choose open neighborhoods  of zero sections
\[
V=\{v\in TN_0^\perp:|v|<\de\} \quad\text{and}\quad V^\dagger=\{v\in T(\pa N_0)^\dagger:|v|<\ep/2\}
\]
such that the Riemannian exponential map induces diffeomorphisms
\begin{equation}\label{tau}
\ta:V\subset TN_0^\perp\to U\subset M, 
\end{equation}
and
\begin{equation}\label{tautau}
\ta^\dagger:\{v\in T(\pa N_0)^\dagger:|v|<\ep\}\subset T(\pa N_0)^\dagger\to U^\dagger\subset M.
\end{equation}
{Note that $U$ contains $N_0$ but is not a neighborhood of $N_0$ in $M$ because $N_0$ has boundary.
Note also that, since $N_0$ is totally geodesic, $U^\dagger\cap N_0$ is an open set in $N_0$,
a so called collar neighborhood.
This is crucial for the following construction.}

We will build in a canonical way normal embeddings relative to $N_0$ from given sections $ \sigma^\dagger \in \Gamma _{V^\dagger}( T(\partial N_0)^\dagger) $
and $ \sigma \in \Gamma _V( TN_0^\perp)$,
where $\Gamma _{V^\dagger}( T(\partial N_0)^\dagger)$ and $\Gamma _V( TN_0^\perp)$ denote open subsets of the corresponding Fr\'echet spaces 
$\Gamma ( T(\partial N_0)^\dagger)$ and $\Gamma( TN_0^\perp)$
as in \eqref{gamave}. This is done in several steps.
 
\textbf{Step I.} We extend $N_0$ through its boundary 
to a boundaryless submanifold $N^{ext}_0$ of $M$,
having the same dimension as $N_0$.
This is done by adding to $N_0$ geodesic segments 
starting at $\pa N_0$ in orthogonal direction to $\partial N_0$.
More precisely, $N^{ext}_0=N_0\cup U^\dagger$, where $U^\dagger=\exp(\{v\in T(\pa N_0)^\dagger:|v|<\ep\})$.
{Here we make use of the fact that $N_0$ is totally geodesic}. 
Thus $ \tau ^\dagger$ from \eqref{tautau} is a tubular neighborhood of $\pa N_0$
in $N_0^{ext}$.
We extend now the diffeomorphism $\ta$ from \eqref{tau} 
via the exponential map to a diffeomorphism 
\begin{equation}\label{diffeo_tube} 
\ta^{ext}: V^{ext}\subset(TN^{ext}_0)^\perp 
\rightarrow U^{ext} \subset M.
\end{equation}

\textbf{Step II.} The image of the embedding $\tau^\dagger\o\sigma^\dagger: {\partial N_0 \rightarrow U^\dagger\subset N_0^{ext}}$
is the boundary $\pa N_1$ of a submanifold of $N_1\subset N_0^{ext}$
with $\dim N_0=\dim N_1$.
We will extend this embedding to an embedding $\phi_{\si^\dagger}:N_0\to N_0^{ext}$
with image  $N_1$, 
which modifies $1_{N_0}$ only in the collar neighborhood 
$U^\dagger\cap N_0$.
This is done with the help of a smooth function
$\rho:\RR\to\RR$ which vanishes in a neighborhood of $(-\infty,-1]$,
with $\rho(0)=1$ and bounded derivative: $0\le\rho'(t)\le 2$ for all $t\in\RR$.
The function $\rho$ is chosen such that,
for any $a\in(-\ep/2,\ep/2)$, 
the function $\la_a:t\in[-\ep,0]\mapsto t+a\rho(t/\ep)\in[-\ep,a]$ 
is a monotonically increasing bijection. 
Indeed, $\la_a(-\ep)=-\ep$, $\la_a(0)=a$, and $\la_a'(t)=1+a/\ep\rh'(t/\ep)>1-1/2\rh'(t/\ep)>0$.

The line bundle $T(\pa N_0)^\dagger$
is the normal bundle of $\pa N_0$ in $N_0^{ext}$.
It is a trivial bundle with trivializing isomorphism simply given by 
$T (\partial N_0)^\dagger \rightarrow \partial N_0 \times \mathbb{R}  $, $v_x \mapsto (x,g(v_x, n))$, where $n$ is the unit outward pointing normal vector field to $ \partial N_0$ in $N_0^{ext}$.
Thus we can identify the open neighborhood $V^\dagger$ of $\pa N_0$ with $\pa N_0\x(-\ep/2,\ep/2)$
and the section $\si^\dagger$ with a function $\si^\dagger:\pa N_0\to(-\ep/2,\ep/2)$.
Now, with the help of the embedding 
$$
h:\pa N_0\x(-\ep,0]\to\pa N_0\x(-\ep,\ep),
\quad h(x,t)=(x,t+\si^\dagger(x)\rho(t/\ep)),
$$ 
we can build an embedding $\phi_{\si^\dagger}:N_0\to N_0^{ext}$ with image $N_1$ by 
$$\phi_{\si^\dagger}|_{U^\dagger\cap N_0}=\ta^\dagger\o h\o(\ta^\dagger)^{-1}\text{ and }\phi_{\si^\dagger}|_{N_0-U^\dagger}=1|_{N_0-U^\dagger}.$$

\textbf{Step III.} We use parallel transport along geodesic segments 
starting at $\pa N_0$ in orthogonal direction to $\partial N_0$
to build a bundle isomorphism $H_{\si^\dagger}$ from $TN_0^\perp$ 
to $TN_1^\perp$ over the embedding $\phi_{\si^\dagger}:N_0\to N_0^{ext}$.
The map $H_{\si^\dagger}$ restricted to the fibers over $N_0-U^\dagger$ is the identity,
while its restriction to the fiber over $y\in U^\dagger\cap N_0$ is the parallel transport along the
geodesic segment that connects $y$ and $\phi_{\si^\dagger}(y)$.
We remark that parallel transport along geodesics 
preserves the norm and maps orthogonal vectors into orthogonal vectors.
{In particular $H_{\si^\dagger}$ is uniquely determined by an embedding of the (unit) sphere bundle
$S(TN_0^\perp)$ to the unit sphere bundle $S((TN_0^{ext})^\perp)$,
also denoted by $H_{\si^\dagger}$.}

\textbf{Step IV.} The {\bfi normal embedding relative to $N_0$}, associated to 
the two sections $ \si^\dagger \in \Gamma_{V^\dagger} ( T(\partial N_0)^\dagger)$ 
and $ \sigma \in \Gamma_V (TN_0^\perp)$, is then defined by
\begin{equation}\label{normal_embedding_boundary} 
f^\perp:= \tau^{ext} \circ H_{ \si^\dagger } \circ \sigma\o f_0\in\Emb(S,M) .
\end{equation} 
This means that we first lift $N_0=f_0(S)$ in a perpendicular direction by using $ \sigma $ and get a manifold  whose boundary is exactly sitting above the boundary of $ N_0$. 
Then we apply the diffeomorphism $H_{\si^\dagger}$,
which stretches a collar neighborhood of the boundary
of the new manifold in a parallel direction 
to $N_0$ according to $\si^\dagger $.

\medskip

Let us denote by $ \mathcal{U} $ the set of all submanifolds obtained by $N=f^\perp(S)$
for normal embeddings relative to $N_0$, associated to 
$ \si^\dagger \in \Gamma_{V^\dagger} ( T(\partial N_0)^\dagger)$ 
and $ \sigma \in \Gamma_V (TN_0^\perp)$. 
We now show that, given $N \in \mathcal{U} $, there is a unique normal embedding
relative to $N_0$ whose image is $N$. We first project the boundary of $N$ down to $N^{ext}_0$ using the projection $(TN_0^{ext})^\perp\rightarrow N_0^{ext}$. This yields 
a hypersurface $ \Sigma $ in $N_0^{ext}$ that is contained in $U^\dagger$.
Because $N\in\U$,  $\Sigma$ determines a unique section $ \sigma^\dagger \in \Gamma_{V^\dagger} ( T(\partial N_0)^\dagger)$ such that $\tau^\dagger \circ \si^\dagger(N_0)  =\Sigma $. 
By construction of $H_{\si^\dagger}$, 
the manifold $H_{\si^\dagger} ^{-1} ( N)$  sits exactly above $N_0$: its boundary sits exactly above the boundary of $ N_0$, 
and $ H_{\si^\dagger} $ sends  fibers above $ \partial N_0$ to fibers above $ \Sigma $. 
Now, there is a unique $ \sigma \in\Ga_V(TN_0^\perp)$ that lifts $N_0$  to this manifold $H_{\si^\dagger} ^{-1} ( N)$ above $N_0$.

\medskip

Fr\'echet manifold charts for $ \operatorname{Gr}^S(M)$ can now be built in a similar way as in \S\ref{2.1}. Indeed, a chart around $N_0$ is defined on the subset $\U\subset \operatorname{Gr}^S(M)$ by
\[
\chi : \mathcal{U} \rightarrow \Gamma_{V^\dagger} (T( \partial N_0)^\dagger) \oplus \Gamma_V (TN_0^\perp), \quad \chi(N):=( \sigma ^\dagger, \sigma ),
\]
where $ ( \sigma ^\dagger, \sigma )$ are the unique sections such that the corresponding normal embedding in \eqref{normal_embedding_boundary} verifies $f^\perp(S)=N_0$. 
The inverse is $\chi^{-1}( \sigma ^\dagger, \sigma )=f^\perp(S)$.

To prove that the change of charts $\chi_j\o\chi_i^{-1}$ are smooth, we decompose it in a number of steps that consist in smooth mappings. The two charts are defined with submanifolds $N_i=f_i(S)$, $N_j=f_j(S)$, 
and Riemannian metrics $g_i,g_j$ on $M$ that induce diffeomorphisms 
$\ta_i^{ext},\ta_j^{ext},\ta_i^\dagger,\ta_j^\dagger$ via exponential maps.
For $(\si^{\dagger} ,\si ) \in \Gamma_{V_i ^\dagger} (T( \partial N_i)^\dagger) \oplus \Gamma_{V_i } (TN_i^\perp)$, these steps are
\begin{align*}
\si ^\dagger&\mapsto H _{i,\si ^\dagger}\in\Emb(S(TN_i^\perp),S((TN_i^{ext})^\perp)),\\
(\si^{\dagger},\si )&\mapsto f^{\perp_i }(\si^{\dagger},\si)=\ta_i^{ext}\o 
H_{i,\si^\dagger}\o\si\o f_i\in\Emb(S,M),\\
f^{\perp_i }&\mapsto E_j (f^{\perp_i}):= (\ta_j^{ext})^{-1}\o f^{\perp_i}\o f_j^{-1}\in\Emb(N_j,(TN_j^{ext})^\perp),\\
(\si^\dagger,\si)&\mapsto (\si^\dagger_{new}, \si_{new})=\left( 
(\ta_j^\dagger)^{-1}\o p_{N_j^{ext}}\o E_j |_{\pa N_j},H_{j,\si^\dagger_{new}}^{-1}\o E_j )\right),
\end{align*} 
where $E_j =E_j (f^{\perp_i }(\si^{\dagger},\si))$ and $ p_{N_j^{ext}}:(TN_j^{ext})^\perp\to N_j^{ext}$.

The principal $\Diff(S)$-bundle structure on $ \pi:\operatorname{Emb}(S,M)\to\Gr^S(M)$ is now constructed as in the case of a manifold $S$ without boundary in \S\ref{s2.1}.

We have thus proved the following result.

\begin{theorem}\label{thm_general}  Let $S$ be a smooth compact manifold with smooth boundary and let $M$ be a finite dimensional manifold without boundary. Suppose that $ \operatorname{dim}S\leq \operatorname{dim}M$. Then the nonlinear Grassmannian $ \operatorname{Gr}^S(M)$ is a smooth Fr\'echet manifold. {The connected component of $N \in  \operatorname{Gr}^S(M)$ is modeled on the Fr\'echet vector space $\Gamma (T( \partial N)^\dagger) \oplus \Gamma (TN^\perp)$}.

Moreover the projection
\begin{equation}\label{proj_general} 
\pi :\operatorname{Emb}(S,M) \rightarrow\operatorname{Gr}^S(M), \quad f \mapsto f(S)
\end{equation} 
is a Fr\'echet principal $ \operatorname{Diff}(S)$-bundle. 
\end{theorem}

\paragraph{The tangent map.}
A tangent vector $v_f\in T_f\Emb(S,M)=\Ga(f^*TM)$ 
defines a unique section $v_N:=v _f \circ f ^{-1} \in\Ga(TM|_N)$ for $N=f(S)$. Using this notation, the tangent map to the projection $ \pi $ reads $T_f\pi(v _f )= (v_N^\dagger,v_N^\perp) \in T_N \operatorname{Gr}^S(M)=\Gamma (T( \partial N)^\dagger) \oplus \Gamma (TN^\perp)$,
where $v_N^\dagger$ is the orthogonal projection of $v_N|_{\pa N}$
on the line bundle $T(\pa N)^\dagger\subset TN|_{\pa N}$ 
orthogonal to the boundary $\pa N$, and $v_N^\perp$ is the orthogonal projection
of $v_N\in TM|_{N}$ on the orthogonal bundle $TN^\perp$. Note that by using the metric $g$, we have the identification $\Gamma (T( \partial N)^\dagger)= C^\infty( \partial N, \mathbb{R}  )$ and the first term can be written as $g(v_N|_{ \partial N}, n)$, where $n$ is the unit normal outward pointing vector field to $ \partial N$.

The first component vanishes when $S$ has no boundary, as seen in \S\ref{s2.1},
while the second component vanishes when $S$ and $M$ have the same dimension,
as we will see in more details in the next paragraph.

Note also that the vertical subspace of $T_f \operatorname{Emb}(S,M)$ reads
\[
V_f \operatorname{Emb}(S,M)=\{Tf \circ  u : u \in \mathfrak{X}_{\|}(S)\}= \{ v \circ f : v \in \mathfrak{X}  _\|( f(S))\},
\]
where $ \mathfrak{X}  _\|(S)$ denotes the space of smooth vector fields on $S$ parallel to the boundary. In the equality above, we defined $v:= f _\ast u$. Since $u$ is an arbitrary vector field on $S$ parallel to the boundary, $v$ is an arbitrary vector field on $f(S)$ parallel to the boundary.

{We consistently have an isomorphism $T_f \operatorname{Emb}(S,M)/V_f \operatorname{Emb}(S,M)\simeq T_N \operatorname{Gr}^S(M)$, for $N=f(S)$.}

A metric is not needed to identify the tangent space, since we can make the identification $T_N \operatorname{Gr}^S(M)= \Gamma (TN|_{\pa N}/T \partial N) \times \Gamma (TM|_N/TN)$, in which case we can write $T_f \pi (v _f )= \left( [ v _f \circ f ^{-1} |_{\partial N}],[v _f \circ f ^{-1}]\right) $.

\paragraph{Embeddings of manifolds of the same dimension.}
In the special case $\dim S=\dim M$
parts of the construction of a normal embedding relative to $N_0$
simplify because the normal bundle over $N_0$ is zero. 
The normal embedding is associated to a single section,
namely the section  $\si^\dagger$ of the bundle 
$T(\pa N_0)^\dagger= T(\pa N_0)^\perp$ over $\pa N_0$.
The Riemannian metric $g$ on $M$ can be arbitrarily chosen
and the manifold $N_0^{ext}$ that extends $N_0$
is just an open subset of $M$.
The Fr\'echet manifold $\Gr^S(M)$ is modeled on $\Gamma(T( \partial N_0)^\dagger)= \Gamma (T(\pa N_0)^\perp)$. Rewritten in this case, Theorem \ref{thm_general} reads as follows.

\begin{proposition} Let $S$ be a smooth compact manifold with smooth boundary and let $M$ be a finite dimensional manifold without boundary. Suppose that $ \operatorname{dim}S=\operatorname{dim}M$. Then the nonlinear Grassmannian $ \operatorname{Gr}^S(M)$ is a smooth Fr\'echet manifold.  
{The connected component of $N \in  \operatorname{Gr}^S(M)$ is modeled on the Fr\'echet vector space $\Gamma (T( \partial N)^\perp)\simeq C^\infty( \partial N)$}.

Moreover the projection
\begin{equation}\label{proj_S=M} 
\pi :\operatorname{Emb}(S,M) \rightarrow\operatorname{Gr}^S(M), \quad f \mapsto f(S)
\end{equation} 
is a Fr\'echet principal $ \operatorname{Diff}(S)$-bundle. 
\end{proposition} 

Note that the tangent map to $ \pi $ read $T_f \pi (v _f )= v_N^\dagger\in \Gamma (T( \partial N)^\perp)$, where $v_N:=v _f \circ f ^{-1} $ and $v_N^\dagger$ is the orthogonal projection of $v_N|_{\pa N}$ to the line bundle $T(\pa N)^\perp$. It can also be written as $T_f \pi (v _f )= g( v_N|_{ \partial N}, n)\in C^\infty(\partial N)$, where $g$ is the Riemannian metric on $M$ and $n$ is the unit normal outward pointing vector field to $ \partial N$.
\color{black}


\section{Principal $\operatorname{Diff}(S)$-bundle structure on embeddings that preserve the total volume}\label{s.totvol}

In this section we consider the subset $ \operatorname{Emb}_0(S,M) \subset \operatorname{Emb}(S,M)$ of all embeddings $f:S \rightarrow M$ that preserve the total volume. This set arises as the total space of a $ \operatorname{Diff}(S)$-principal bundle over the nonlinear Grassmannian $\operatorname{Gr}^S_0(M)$ of submanifolds $N$ diffeomorphic to $S$, with same volume with $S$.

We first consider below the case $ \operatorname{dim}S= \operatorname{dim}M$ since it is considerably simpler. For the case $ \operatorname{dim}S < \operatorname{dim}M$ a Riemannian metric is needed on $M$.

\subsection{The case $ \operatorname{dim}S= \operatorname{dim}M$}
 
We fix volume forms $ \mu$ on $S$ and $ \mu_M$ on $M$.

\begin{proposition}
The set of embeddings that preserve the total volume
\[
\operatorname{Emb}_0(S,M) :=\left\{ f \in \operatorname{Emb}(S,M) :\int_{f(S)} \mu_M  =\int_S \mu \right\}  
\]
is a splitting submanifold of $ \operatorname{Emb}(S,M)$, with tangent space
\[
T_f \operatorname{Emb}_0(S,M)= \left\{ v _f \in T_f \operatorname{Emb}(S,M) : \int_{ \partial N} i_{ \partial N} ^\ast \mathbf{i} _{ v _f \circ f ^{-1} } \mu _M = 0\right\}.
\]
\end{proposition}

\begin{proof}
Consider the smooth map
$\operatorname{vol}  :\operatorname{Emb}(S,M) \rightarrow \mathbb{R}$, $\operatorname{vol}  (f):= \int_S f ^\ast \mu_M$
with derivative
\[
\mathbf{d} \operatorname{vol}( f ) \cdot v_ f = \int_S \mathbf{d} f ^\ast \mathbf{i} _{ v_ f\circ f^{-1}} \mu_M 
= \int_{ \partial S} i_{\pa S}^* f ^\ast \mathbf{i} _{ v_ f {\circ f ^{-1}}} \mu_M{=\int_{ \partial N} i_{ \partial N} ^\ast \mathbf{i} _{ v _f \circ f ^{-1} } \mu _M},
\]
where $ i _{ \partial S}: \partial S \rightarrow S$ is the inclusion.
This map is a submersion with values in a finite dimensional manifold,
so the regular value theorem \cite{NW} mentioned in the Introduction
can be applied to the submersion $\vol$ to show that the 
set of embeddings that preserve the total volume $\Emb_0(S,M)=\vol^{-1}(\int_S\mu)$
is a submanifold of $\Emb(S,M)$
of codimension one.
\end{proof}

\medskip
Note that the kernel can be equivalently written
\[
\operatorname{ker} \left( \mathbf{d} \operatorname{vol}( f )\right) = \left\{ v_ f \in T_ f \operatorname{Emb} (S,M)  : i_{ \partial S} ^\ast f ^\ast \mathbf{i} _{ v_ f{\circ f ^{-1}} } \mu_M 
\in \Om_{exact}^{k-1}( \partial S) \right\},
\]
where $k= \operatorname{dim}S= \operatorname{dim}M$. 
When the volume form $ \mu _M$ is associated to a Riemannian metric $g$ on $M$, then we can write $ \mathbf{i} _{ v _f \circ f ^{-1} } \mu _M= g( v _f \circ f ^{-1} , n) \mu _{ \partial N}$, where $n$ is the unit outward pointing normal vector field to $ \partial N$ and $ \mu _{ \partial N}$ is the volume form induced on the boundary. In this case, the tangent space $T_f \operatorname{Emb}_0(S,M)$ consists of vector fields $v _f   \in \Gamma ( f ^\ast TM)$ such that
\[
\int_{ \partial N} g\left(  v _f \circ f ^{-1} |_{ \partial N}, n\right)  \mu _{ \partial N}=0.
\]
\color{black}

In the same way, using the submersion $\vol:\Gr^S(M)\to\RR$, 
$\vol(N)=\int_N\mu_M$, we obtain that the nonlinear Grassmannian of submanifolds of $M$ of same volume as $S$
\begin{equation}\label{defi}
\operatorname{Gr}_0^S(M):=\left \{N\subset M: \text{$N$ submanif. diffeom. to $S$, $\int_N{ \mu _M }= \int_S{ \mu }$}\right \}
\end{equation}
is a codimension one submanifold of $\Gr^S(M)$, with tangent space
\[
T_N \operatorname{Gr}^S_0(M)=\left\{ w _N \in \Gamma (T( \partial N)^\perp): \int_{ \partial N}  i_{\pa N} ^\ast \mathbf{i} _{ w_ N } \mu_M=0\right\}.
\]
Indeed, the derivative is
\[
\mathbf{d} \operatorname{vol}( N) \cdot w_ N 
= \int_N \mathbf{d}i_N^*  \mathbf{i} _{ w_ N} \mu_M  
= \int_{ \partial N}  i_{\pa N} ^\ast \mathbf{i} _{ w_ N } \mu_M,
\]
for all $w_N\in\Ga(TN^\perp)$.

{As above, when $ \mu _M$ is associated to a Riemannian metric $g$ on $M$, a section $w_N\in \Gamma (T( \partial N)^\perp)$ belongs to the tangent space $T_N \operatorname{Gr}^S_0(M)$ if and only if $\int_Ng(w_N,n) \mu _{ \partial N}=0$. The tangent space can thus be identified with $C_0^\infty( \partial N)=\left \{ h \in C^\infty(N, \mathbb{R}  )\mid \int_N h \mu _{ \partial M}=0\right\}$, where $h= g(w_N,n)$, i.e., $w_N=hn$.}

We thus have proved the following theorem.
 
\begin{proposition} The projection
\[
\pi _0 :\operatorname{Emb}_0 (S,M) \rightarrow\operatorname{Gr}^S_0 (M), \quad f \mapsto f(S)
\]
is a Fr\'echet principal $ \operatorname{Diff}(S)$-bundle. 
\end{proposition}

The projection $\pi _0 $ is the restriction of the projection $ \pi $ in \eqref{proj_S=M}, so the expression of its tangent map is the restriction of that of $ \pi $.

\subsection{The case $ \operatorname{dim}S< \operatorname{dim}M$}\label{s<}

We fix a volume form $ \mu$ on $S$ and a Riemannian metric $g$ on $M$
with induced volume form $ \mu(g)$. 
This metric induces on any submanifold $N$ of $M$ a volume form $\mu(g_N)$, where $g _N $ is the Riemannian metric induced on $N$. We denote by $H_N\in\Ga(TN^\perp)$
the trace of the second fundamental form 
$\operatorname{II}_N:TN\x TN\to TN^\perp$
of the submanifold $N$, \ie the mean curvature vector field.

\begin{proposition}
The set of embeddings that preserve the total volume
\[
\operatorname{Emb}_0(S,M) :=\left\{ f \in \operatorname{Emb}(S,M) :\int_{f(S)} \mu(g_{f(S)})  =\int_S \mu \right\}  
\]
is a splitting submanifold of $ \operatorname{Emb}(S,M)$, whose tangent space at $f$ is
\[
\left\{ v _f \in T_f \operatorname{Emb}(S,M) : \int_{ \partial N}g ( v _f \circ f ^{-1} , n) \mu (g_{\partial N})=\int_Ng(H_N, v _f \circ f ^{-1}) \mu (g_N)\right\}.
\]
\end{proposition}

\begin{proof}
Consider the smooth map $\operatorname{vol}  :\operatorname{Emb}(S,M) \rightarrow \mathbb{R}$, $\operatorname{vol}  (f):= \int_S f ^\ast \mu(g_{f(S)} )$
with derivative
\[
\mathbf{d} \operatorname{vol}( f ) \cdot v_ f 
= \int_{ \partial S}  i_{ \partial S} ^\ast f ^\ast \mathbf{i} _{ (v_ f\circ f ^{-1})^\top  } \mu(g_{f(S)})
-\int_S g(H_{f(S)}\o f,v_f)f ^\ast \mu(g_{f(S)}),
\]
see e.g. \cite{GaHuLa2004} (the proof of Theorem 5.20). This is a submersion 
with values in a finite dimensional manifold
for which we can apply 
the regular value theorem \cite{NW} mentioned in the Introduction to show that the 
set of embeddings that preserve the total volume $\Emb_0(S,M)=\vol^{-1}(\int_S\mu)$
is a submanifold of $\Emb(S,M)$
of codimension one.
\end{proof}

\medskip

In the same way, using the submersion $\vol:\Gr^S(M)\to\RR$, $\vol(N)=\int_N\mu(g_N)$,
with derivative
\begin{equation}\label{derivative_vol} 
\mathbf{d} \operatorname{vol}( N ) \cdot (w_{\pa N},w_ N) 
= \int_{ \partial N}{g(w_{ \partial N}, n)\mu(g_{\partial N})}
-\int_Ng(H_{N},w_N)\mu(g_{N}) ,
\end{equation} 
for $(w_{\pa N},w_N)\in T_N\Gr^S(M)=\Ga(T(\pa N)^\dagger)\oplus\Ga(TN^\perp)$, we obtain that the nonlinear Grassmannian
\begin{equation}\label{defi}
\operatorname{Gr}_0^S(M):=\left \{N\subset M: \text{$N$ submanif. diffeom. to $S$, $\int_N{ \mu(g _N) }= \int_S{ \mu }$}\right \}
\end{equation}
of all submanifolds of $M$ of same volume as $S$, 
is a codimension one submanifold of $\Gr^S(M)$. From \eqref{derivative_vol}, the tangent space to $ \operatorname{Gr}^S_0(M)$ reads
\begin{equation}\label{tgt_space_Gr0} 
T_N\operatorname{Gr}^S_0(M)=\left \{(w_{ \partial N}, w_N) :\int_{ \partial N} {g(w_{ \partial N}, n)\mu(g_{\partial N})}=\int_Ng(H_{N},w_N)\mu(g_{N}) \right \}.
\end{equation} 
As in the previous case we get the following result.

\begin{proposition} The projection
\[
\pi _0 :\operatorname{Emb}_0 (S,M) \rightarrow\operatorname{Gr}^S_0 (M), \quad f \mapsto f(S)
\]
is a Fr\'echet principal $ \operatorname{Diff}(S)$-bundle. 
\end{proposition}

The projection $\pi _0 $ is the restriction of the projection $ \pi $ in \eqref{proj_general}, so the expression of its tangent map is the restriction of that of $ \pi $, namely $T_f \pi _0 (v _f )=(  v_N^\dagger, v_N^\top )=(w_{ \partial N}, w _N )\in T_N \operatorname{Gr}^S_0(M)$, see \eqref{tgt_space_Gr0}, where $v_N:= v _f \circ f ^{-1}$.

Note that $V_f \operatorname{Emb}_0(S,M)= V_f \operatorname{Emb}(S,M)$, for all $ f \in \operatorname{Emb}_0(S,M)$, and    
that we consistently have $T_f\operatorname{Emb}_0(S,M)/ V_f \operatorname{Emb}_0(S,M)\simeq T_N \operatorname{Gr}_0^S(M)$. 


\section{Principal $\Diff_{\vol}(S)$-bundle structure on embeddings}\label{s2}

In this section we consider another principal bundle structure on $\operatorname{Emb}(S,M)$, that naturally arises in \cite{GBVi2014} in the context of symplectic reduction applied to the dual pair of momentum maps for the ideal fluid.
{Here we assume that $S$ is compact, possibly with boundary, and $\dim S\le\dim M$.}

\paragraph{Nonlinear Grassmannian of volume submanifolds.} 
Let us assume that $S$ is endowed with a volume form $ \mu $.
This fixes an orientation on $S$ that we use whenever we integrate over $S$.
The \textit{nonlinear Grassmannian of volume submanifolds of type $(S,\mu)$} is
\begin{equation}\label{def_GR_S_mu} 
\Gr^{S,\mu}(M):=\left\{(N,\nu):N\in\Gr^{S}(M),\;\nu\in\Vol(N),\;\int_N\nu=\int_S\mu\right\},
\end{equation} 
where the orientation on $N$ is the one induced by the volume form $\nu$.

We consider the set of volume forms 
\begin{equation}\label{volmu}
\Vol_\mu (S)=\left\{ \rho\in\Vol(S) : \int_S \rho =\pm \int_S \mu \right\},
\end{equation} 
union of two convex connected components.
The fiber of the forgetting map 
\begin{equation}\label{forg}
\pi_1:\Gr^{S,\mu}(M)  \rightarrow  \Gr^S(M), \quad (N, \nu )\mapsto \pi _1 (N, \nu )=N
\end{equation}
is isomorphic to $\Vol_{ \mu }(S)$. Indeed, the fiber of  $\Gr^{S, \mu }(M)$ above $N$ is $ \{ \nu \in \Vol(N):\int_N \nu =\int_S \mu \}$, with the orientation on $N$  induced from $ \nu $. This means that 
if $\nu$ belongs to the fiber, then $-\nu$ is in the fiber too,
so that the fiber is  isomorphic to $\Vol_\mu(S)$.

Next we show that $\Gr^{S,\mu}(M)$ can be identified with the associated bundle 
to the principal $ \operatorname{Diff}(S)$-bundle $\Emb(S,M)\to\Gr^S(M)$
with respect to the natural action of the diffeomorphism group
on $\Vol_\mu(S)$.
This is made precise in the following proposition.

\begin{proposition}\label{fiber_bundle} 
The nonlinear Grassmannian of volume submanifolds of type $(S,\mu)$ 
with the forgetting map
$\pi_1:\Gr^{S,\mu}(M)  \rightarrow  \Gr^S(M)$
is a Fr\'echet fiber bundle with fiber $\Vol_\mu(S)$, that can be identified with the associated bundle
\[
\Emb(S,M)\x_{\Diff(S)}\Vol_\mu(S)\rightarrow \operatorname{Gr}^S(M), 
\quad [f, \rho  ] \mapsto f(S)=N.
\]
In particular, $\Gr^{S,\mu}(M)$ is a Fr\'echet manifold. 
{The connected component of $(N, \nu ) \in \Gr^{S,\mu}(M)$ is modeled on the Fr\'echet vector space $ \Gamma (T(\partial N)^\dagger) \oplus \Gamma ( TN^\perp) \oplus { \Omega ^k _0(N)}$}, where 
\begin{equation}\label{k_form_0} 
\Omega ^k _0(N)  :=  \left \{\sigma  \in  \Omega ^{k}(N):\int_N \sigma  =0\right \}, \quad k= \operatorname{dim}S. 
\end{equation} 
\end{proposition}

\begin{proof}
The fiber bundle isomorphism is given by
\begin{equation}\label{biso}
(N,\nu)\in\Gr^{S,\mu}(M)\mapsto [f,f^*\nu]\in \Emb(S,M)\x_{\Diff(S)}\Vol_\mu(S),
\end{equation}
where $f$ is any embedding with image $N$.
We first check that $f^*\nu\in\Vol_\mu(S)$. Indeed,
$\int_Sf^*\nu=\pm\int_{f(S)}\nu$,
since the orientation on $f(S)=N$ is induced by the volume form $\nu$,
which might or might not coincide with that induced by $f_*\mu$
(recall that the integral over $S$ uses the orientation induced by $\mu$). 
But $\int_N\nu=\int_S\mu$, so that $\int_Sf^*\nu=\pm\int_S\mu$.
Then we verify that the map is well defined, \ie
the result doesn't depend on the choice of the embedding with image $N$:
$[f\o\ph,(f\o\ph)^*\nu]=[f\o\ph,\ph^*f^*\nu]=[f,f^*\nu]$
for any $\ph\in\Diff(S)$.

The inverse of \eqref{biso} is given by
$$[f,\rho]\in\Emb(S,M)\x_{\Diff(S)}\Vol_\mu(S)\mapsto (f(S),f_*\rho)\in\Gr^{S,\mu}(M).$$
We have 
$0<\int_{f(S)}f_*\rho=\pm\int_S\rho$
with orientation on $f(S)$ induced by the volume form $f_*\rho$,
which might or might not coincide with that induced by $f_*\mu$.
But $0<\int_S\mu=\pm\int_S\rho$, so $\int_{f(S)}f_*\rho=\int_S\mu$,
which shows that $(f(S),f_*\rho)\in\Gr^{S,\mu}(M)$.
It remains to verify that the map doesn't depend on the choice of the representative
in the associated bundle:
$[f\o\ph,\ph^*\rho]\mapsto((f\o\ph)(S), (f\o\ph)_*\ph^*\rho)=(f(S),f_*\rho)$
for any $\ph\in\Diff(S)$.

Using the principal bundle charts for $ \operatorname{Emb}(S,M)$, it is now standard to build bundle manifold charts for the associated bundle. Indeed, we have the isomorphisms
\[
\pi_1^{-1}(\U)=\pi^{-1}(\U)\x_{\Diff(S)}\Vol_\mu(S)\cong(\U\x\Diff(S))\x_{\Diff(S)}\Vol_\mu(S)
\cong\U\x\Vol_\mu(S).
\]
This shows that $\pi_1:\Gr^{S,\mu}(M)  \rightarrow  \Gr^S(M)$ is a smooth Fr\'echet fiber bundle. {Since $\Gr^{S,\mu}(M)$ is an associated bundle, its model Fr\'echet space is isomorphic to the sum of the model Fr\'echet space of $\Gr^{S}(M)$ and the model Fr\'echet space of $\operatorname{Vol}_ \mu (S)$, i.e., $\Gamma (T(\partial N)^\dagger) \oplus \Gamma ( TN^\perp)$ and $\Omega ^k _0(N)$}. \end{proof}

\paragraph{Decomposition of the group of diffeomorphism.} 
A theorem of Moser ensures,
for a given  volume form $ \mu $ on $S$, 
the existence of a smooth map $B:\Vol_\mu(S)\to \Diff(S)$
defined on the space \eqref{volmu} of volume forms on $S$ with total volume $\int_S \mu $, such that $B(\nu)_*\mu=\nu$.

Using this result one shows that the group of diffeomorphisms of $S$ splits smoothly
as the product of the group of volume preserving diffeomorphisms
and $\Vol_\mu(S)$:
\begin{equation}\label{mose}
\Diff(S)=\Diff_{\vol}(S)\x\Vol_{{ \mu }} (S).
\end{equation}
This result is valid also for manifolds with smooth boundary, see Theorem 5.1 and 8.6 in \cite{EbMa1970}. 
Recall that the diffeomorphism \eqref{mose} reads $\ph\mapsto(\ph_{\vol},\ph_*\mu)=(B(\ph_*\mu)^{-1}\o\ph,\ph_*\mu)$
with inverse $(\ps,\nu)\mapsto B(\nu)\o\ps$.

The version of \eqref{mose} that appears in  \cite{EbMa1970} is
$\Diff_+(S)=\Diff_{\vol}(S)\x\Vol_{{ \mu }} ^+(S)$,
where $\Diff_+(S)$ denotes the group of orientation preserving diffeomorphisms
and $\Vol_{{ \mu }}^+ (S)=\left \{\rho\in\Vol(S):\int_S\rho=\int_S\mu\right\}$. 

Therefore, every diffeomorphism $\ph\in\Diff(S)$ can be transformed into
a volume preserving diffeomorphism
\begin{equation}\label{bvol}
\ph_{\vol}:=B(\ph_*\mu)^{-1}\o\ph\in\Diff_{\vol}(S).
\end{equation} 
There is another possibility to obtain a volume preserving diffeomorphism
out of an ordinary diffeomorphism $\ph$, namely, by composing on the right with $B(\ph^*\mu)$.

\paragraph{Principal $ \operatorname{Diff}_{\vol}(S)$-bundle structure on $\operatorname{Emb}(S,M)$.}
In the same way as $\pi:f\in\Emb(S,M)\mapsto f(S)\in\Gr^{S}(M)$ is a principal $\Diff(S)$-bundle, we now show that the map
\begin{equation}\label{pinu}
\pi_{\vol}:\Emb(S,M)\rightarrow \Gr^{S,\mu}(M), \quad f \mapsto \pi _{\vol}(f)= (f(S),f_*\mu),
\end{equation}
is a principal $\Diff_{\vol}(S)$-bundle. 

In order to prove that $ \pi_{\vol} $ is surjective, given an arbitrary $(N,\nu)\in \Gr^{S,\mu}(M)$, we need to find $f\in \operatorname{Emb}(S,M)$ such that $f(S)=N$ and $f_* \mu= \nu $.
We know that there exists $f'\in \operatorname{Emb}(S,M)$ such that $N=f'(S)$. Let us consider $(f')^*\nu\in \Vol(S)$. Since $\int_S \mu=\int_N \nu=\int_S (f')^*\nu$, there exists $ \varphi =B((f')^*\nu)\in \operatorname{Diff}(S)$ such that $ \varphi _* (f')^*\nu= \mu $ (by Moser's theorem).
Then the requested embedding is $f:= f'\circ \varphi^{-1} $.
On the other hand, if the embeddings $f_1$ and $f_2$ satisfy $\pi_{\vol}(f_1)=\pi_{\vol}(f_2)$,
then there is a diffeomorphism $\ph$ of $S$ such that $f_2=f_1\o \ph$,
while the identity $(f_1)_*\mu=(f_2)_*\mu$ ensures that the diffeomorphism $\ph$ is volume preserving.
For the moment we get a bijection:
\[
\operatorname{Emb}(S,M)/ \operatorname{Diff}_{\vol}(S) \longleftrightarrow   \operatorname{Gr}^{S,\mu}(M), \quad [f] \mapsto (f(S),f_*\mu).
\]

\paragraph{Local trivializations.} They can be built
over the open sets $\V=\pi_1^{-1}(\U)\subset \Gr^{S,\mu}(M)$,
preimages by the forgetting map \eqref{forg}
of the open sets $ \mathcal{U} $ of all submanifolds of $M$ that are images of normal embeddings. 
We take
\[
\Psi:  \pi_{\vol}^{-1}(\mathcal{V}) \rightarrow \mathcal{V}\x\Diff_{\vol}(S), \quad 
\Psi(f)=( \pi_{\vol}(f) , (\psi_f)_{\vol}),
\]
where $ (\psi_f)_{\vol}:=B(\psi_{f*}\mu)^{-1}\circ \psi_f$
with $\psi_f$ from \eqref{psif}.
Its inverse reads
$$\Psi ^{-1} ( (N,\nu) , \ph)= f^\perp \circ B((f^\perp)^*\nu)\circ\ph,$$
where $f^\perp$ is the unique normal embedding such that $f^\perp(S)=N$.
Indeed,
\begin{align*}
\Psi( f^\perp \circ B((f^\perp)^*\nu)\circ\ph)&=((f^\perp(S),f^\perp_*
B((f^\perp)^*\nu)\ph_*\mu),(B((f^\perp)^*\nu)\o\ph)_{\vol})\\
&=((f(S),f^\perp_*B((f^\perp)^*\nu)\mu),B(B((f^\perp)^*\nu)_*\ph_*\mu)^{-1}
\o B((f^\perp)^*\nu)\o\ph)\\
&=((f(S),\nu),\ph)
\end{align*}
and
\begin{align*}
\Psi^{-1}( \pi_{\vol}(f) ,  (\psi_f)_{\vol})=f^\perp\o B((f^\perp)^*\nu)\o B((\psi_f)_*\mu)^{-1}\o\psi_f=f^\perp\o\psi_f=f.
\end{align*}
The transition functions are smooth since they express as:
\[
\phi_{ij}:\V_i\cap\V_j\to\Diff_{\vol}(S),\quad
\phi_{ij}(N,\nu)=(f^{\perp_i})^{-1}\o f^{\perp_j}.
\]

We have thus proved the following result.

\begin{proposition} Let $S$ be a compact manifold, possibly with boundary, 
and let $M$ be a finite dimensional manifold without boundary.
Let $ \mu $ be a volume form on $S$. Then, the projection \eqref{pinu}
\[
\pi _{\vol}: \operatorname{Emb}(S,M) \rightarrow \operatorname{Gr}^{S, \mu}(M)  
\]
is a Fr\'echet principal $ \operatorname{Diff}_{\vol}(S)$-bundle.
\end{proposition} 

Note that the vertical subspace of $T_f \operatorname{Emb}(S,M)$, relative to the $ \operatorname{Diff}_{\vol}(S)$-principal bundle structure, reads
\begin{equation}\label{vertic}
V_f \operatorname{Emb}(S,M)=\{Tf \circ  u : u \in \mathfrak{X}_{\|}(S, \mu )\}= \{ v \circ f : v \in \mathfrak{X}  _\|( f(S), f _\ast \mu )\},
\end{equation}
where $ \mathfrak{X}  _\|(S, \mu )$ denotes the space of smooth vector fields on $S$ parallel to the boundary and divergence-free relative to $ \mu $. In the equality above, we defined $v:= f _\ast u$. Since $u$ is an arbitrary vector field on $S$ parallel to the boundary and divergence-free relative to $ \mu $, $v$ is an arbitrary vector field on $f(S)$ parallel to the boundary and divergence-free relative to the volume for $ f _\ast \mu $ on $f(S)$.

Let $\Om^{k-1}_n(N)=\{\al\in\Om^{k-1}(N):i_{\pa N}^*\al=0\}$
denote the space of differential $(k-1)$-forms normal to the boundary.
Then the space of $ \Omega ^k _0(N)$ of $k$-forms with zero integral, see \eqref{k_form_0}, can be written as $\Om^k_0(N)=\mathbf{d}\Om^{k-1}_n(N)$ \cite{EbMa1970}. Therefore, the tangent space
$T_{(N,\nu)}\Gr^{S,\mu}(M)$ can be identified
with the product space $\Ga(T(\pa N)^\dagger)\x\Ga(TN^\perp)\x \mathbf{d} \Om_n^{k-1}(N)$, for any $(N,\nu)\in\Gr^{S,\mu}(M)$,
by Proposition \ref{fiber_bundle}.

The next proposition gives another expression for the tangent space $T_{(N,\nu)}\Gr^{S,\mu}(M)$,
which permits to express the tangent map $T_f\pi_{\vol}$ in a simple way.

\begin{proposition}\label{deco}
The tangent space to the nonlinear volume Grassmannian
can be identified with
\[
T_{(N,\nu)}\Gr^{S,\mu}(M)={E_{(N, \nu )}}\x\Ga(TN^\perp)
\subset
\Ga(T(\pa N)^\dagger)\x\mathbf{d}\Om^{k-1}(N)\x\Ga(TN^\perp),
\]
where
\[
{E_{(N, \nu )}}:=\{(w_{\pa N},\dd\al)\in\Ga(T(\pa N)^\dagger)\x\mathbf{d}\Om^{k-1}(N)
:w_{\pa N}\nu_{\pa}=i_{\pa N}^*\al\}.
\]

Let $g$ be a Riemannian metric on $M$.
Given $f\in\Emb(S,M)$ such that 
$f(S)=N$ and $f_*\mu=\nu$,  the tangent map $T_f\pi_{\vol}$ becomes
\begin{gather*}
T_f\pi_{\vol}: T_f\Emb(S,M)\to
T_{(N,\nu)}\Gr^{S,\mu}(M),\quad
T_f\pi_{\vol}(v_f)=\left(v_N^\dagger,\pounds _{v_{N}^\top}\nu,v_{N}^\perp\right).
\end{gather*}
Here $v_N:=v_f\o f^{-1}\in\Ga(TM|_N)$ and $v_N=v_N^\perp+v_N^\top$ denotes 
the orthogonal decomposition along $N\subset M$,
while $v_N^\dagger\in\Ga(T(\pa N)^\dagger)$
denotes the orthogonal projection of $v_N^\top|_{\pa N}$ to
the line bundle $T(\pa N)^\dagger$.
\end{proposition}

\begin{proof}
Since $\pi_{\vol}:\Emb(S,M)\to\Gr^{S,\mu}(M)$ is a  principal $\Diff_{\vol}(S)$-bundle, 
we have the identification
$$
T_{(N,\nu)}\Gr^{S,\mu}(M)=T_f\Emb(S,M)/V_f\Emb(S,M),
$$ 
where $V_f\Emb(S,M)$ is the vertical space \eqref{vertic} at $f\in\Emb(S,M)$.
Let $N=f(S)$ and $\nu=f_*\mu$.
The map
\begin{equation}\label{fimu}
v_f\in T_f\Emb(S,M)\;\longmapsto\; \left(v_N^\dagger,\pounds _{v_N^\top}\nu, v_N^{\perp}\right) \in {E_{(N, \nu )}}\x\Ga(TN^\perp)
\end{equation}
is well defined because $\pounds _{v_N^\top}\nu= \mathbf{d}( \mathbf{i} _{v_N^{\top}} \nu) $ and the formula
\begin{equation}\label{formula}
i_{\pa N}^*(\mathbf{i}_{v_N^\top}\nu)= g(v_N^\top, n) \nu _{ \partial }= v_N^\dagger\nu _ \partial,
\end{equation}
which shows that $\left(v_N^\dagger,\pounds _{v_N^\top}\nu\right)  \in {E_{(N, \nu )}}$.

We now show it is a linear surjective map with kernel $V_f\Emb(S,M)$.
Let $v_f\in\Ga(f^*TM)$ be in the kernel of \eqref{fimu}, 
so that $v_N^\dagger=0$, $v_N^\perp=0$ and $\pounds _{v_N^\top}\nu=0$.
This means that $v_N=v_N^\top\in\mathfrak{X}_{\|}(N, \nu )$. Its pull-back 
by the diffeomorphism $f:S\to f(S)=N$, denoted by $u$, belongs to $\mathfrak{X}_{\|}(S, \mu )$.
Hence $v_f=Tf\o u\in V_f{\Emb(S,M)}$.

For the surjectivity we consider a triple $(w_{\pa N},\dd\al,w_N^\perp)\in {E_{(N, \nu )}}\x\Ga(TN^\perp)$, so there is a relation between the first two components
$w_{\pa N}\nu_{\pa}=i_{\pa N}^*\al$.
The differential form $\al\in\Om^{k-1}(N)$ 
determines the vector field $w_N^\top\in\X(N)$ by $\mathbf{i}_{w_N^\top}\nu=\al$. 
Now we can define
$v_f:=v_N\o f$, where $v_N:=w_N^\perp+w_N^\top$,
because of the formula \eqref{formula}.
\end{proof}

\begin{remark}{\rm
The two descriptions of the tangent space $T_{(N,\nu)}\Gr^{S,\mu}(M)$
are of course isomorphic.
This can be seen directly as follows. Using tubular neighborhoods as in Section \ref{gen_case},
we assign  to each $w_{ \partial N}\in\Ga(T(\pa N)^\dagger)$
a differential form $\be(w_{ \partial N})\in\Om^{k-1}(N)$ that extends 
$w_{ \partial N}\nu_\pa\in\Om^{k-1}(\pa N)$, i.e., $w_{ \partial N}\nu_\pa= i _{ \partial N} ^\ast \be(w_{ \partial N})$.
The isomorphism reads $(w_{ \partial N},\dd\al)\in E_{(N, \nu )} \mapsto (w_{ \partial N},\dd(\al-\be(h)))\in \Ga(T(\pa N)^\dagger) \times \Omega ^k _0 (N)$
with inverse $(w_{ \partial N},\dd\la)\in\Ga(T(\pa N)^\dagger) \times \Omega ^k _0 (N)\mapsto (w_{ \partial N},\dd(\la+\be(w_{ \partial N})))\in E_{(N, \nu )}$.
}
\end{remark}


\section{Principal $\Diff_{\vol}(S)$-bundle structure on volume preserving embeddings}\label{s5}

\subsection{The case $\dim S=\dim M$}

We recall that in this case we require $\pa S\ne\varnothing $.
We fix volume forms $ \mu $ and $ \mu_M$ on $S$ and $M$, respectively, and we define the set of volume preserving embeddings
$$ 
\operatorname{Emb}_{\vol}(S,M):=\{ f\in \operatorname{Emb}(S,M)\mid f ^\ast \mu_M= \mu \}.
$$
The projection $\pi(f)=f(S)$ restricted to $\Emb_{\vol}(S,M)$
takes values in the nonlinear Grassmannian 
$\operatorname{Gr}_0^S(M)$ of all type $S$ submanifolds of $M$ of same volume as $S$, introduced in \eqref{defi}, since $f(S)$ has the same volume as $S$
for all $f\in\Emb_{\vol}(S,M)$.

Using a pullback construction,
we show below that $\Emb_{\vol}(S,M)$ has a principal $\Diff_{\vol}(S)$-bundle structure over $\Gr_0^S(M)$.

\begin{proposition}
The pullback bundle of the principal bundle 
$\pi_{\vol}:\Emb(S,M)\rightarrow \Gr^{S,\mu}(M)$
\eqref{pinu} with structure group  $\Diff_{\vol}(S)$, via the smooth section
\[
\et:\Gr_0^S(M)\to\Gr^{S,\mu}(M),\quad\et(N)=(N,i_N^*\mu_M)
\]
of the $\Vol_\mu(S)$ bundle $\pi_1:\Gr^{S,\mu}(M)  \rightarrow  \Gr^S(M)$ \eqref{forg},
can be identified with 
$$\pi:\Emb_{\vol}(S,M)\to\Gr^S_0(M).$$
{The tangent space is $T_f \operatorname{Emb}_{\vol}(S,M)=\{v _f \in \Gamma ( f ^\ast TM)\mid \operatorname{div}_{ \mu _M }( v _f \circ f ^{-1} )=0\}$.}  
\end{proposition}

\begin{proof}
The condition $\et(N)=\pi_{\vol}(f)$ is equivalent to $f(S)=N$ and $f^*\mu_M=\mu$.
Hence the pullback bundle
$$
\et^*\Emb(S,M):=\{(N,f)\in\Gr_0^S(M)\x\Emb(S,M):\et(N)=\pi_{\vol}(f)\}
$$ 
can be identified with $\Emb_{\vol}(S,M)$,
while the canonical projection on $\Gr_0^S(M)$ becomes $\pi:f\mapsto f(S)$.
\end{proof}

\medskip

As a consequence we also get a natural Fr\'echet manifold structure on $\Emb_{\vol}(S,M)$.

\subsection{The case $ \operatorname{dim}S< \operatorname{dim}M$}

In addition to a volume form $ \mu $ on $S$ (with or without boundary), we fix a Riemannian metric $g$ on $M$. 
As in Section \ref{s<}, we denote by $ \mu(g)$ the Riemannian volume form on $M$ and by $ \mu(g _N)$ the Riemannian volume form induced by the Riemannian metric $g$ restricted to $N$. In this case we define
\[
\operatorname{Emb}_{\vol}(S,M):=\{ f\in \operatorname{Emb}(S,M)\mid f ^\ast \mu(g_{f(S)})= \mu\}.
\]
and
\[
\et:\Gr_0^S(M)\to\Gr^{S,\mu}(M),\quad\et(N):=(N,\mu(g_N))
\]
so we have
\[
\et^*\Emb(S,M):=\{(N,f)\in\Gr_0^S(M)\x\Emb(S,M):\et(N)=\pi_{\vol}(f)\},
\]
which can be identified with $\operatorname{Emb}_{\vol}(S,M)$. This proves that $\operatorname{Emb}_{\vol}(S,M)$ is a Fr\'echet manifold and the total space of a principal $\Diff_{\vol}(S)$-bundle over $\Gr_0^S(M)$, for $ \operatorname{dim}S< \operatorname{dim}M$ and $S$ with or without boundary.
The tangent space is
\[
T_f \operatorname{Emb}_{\vol}(S,M)=\left\{ v _f \in \Gamma ( f ^\ast TM) :\operatorname{div}_{\mu (g_N)}( v _f \circ f ^{-1} )^\top =g(H_N, v _f \circ f ^{-1}) \right\}.
\] 
We notice that the principal bundle
\[
\pi:\Emb_{\vol}(S,M)\to\Gr^S_0(M)
\]
is a reduction of the structure group of the principal bundle 
$\Emb_0(S,M)\to \Gr^S_0(M)$
from $\Diff(S)$ to the subgroup $\Diff_{\vol}(S)$.

{\footnotesize

\bibliographystyle{new}

\bigskip

FRAN{C}OIS GAY-BALMAZ, 
{LMD, Ecole Normale Sup\'erieure/CNRS, Paris, France.\\
\texttt{gaybalma@lmd.ens.fr}
\\

CORNELIA VIZMAN, {Department of Mathematics,
West University of Timi\c soara, 
Romania.\\
\texttt{vizman@math.uvt.ro}
}

\end{document}